\documentclass[a4paper, 12pt,reqno]{amsart}
\usepackage{amsfonts, amsmath, amssymb, amsthm, enumitem}  
\usepackage[english]{babel}
\usepackage[utf8]{inputenc}
\usepackage[OT1]{fontenc}
\usepackage{bm}
\usepackage{textcomp, cmap, comment}	
\usepackage{graphicx, wrapfig}
\usepackage{epigraph, xcolor, hyperref}
\usepackage{array,tabularx,tabulary,booktabs}
\usepackage{euscript, mathrsfs}	

\usepackage[margin=2.5cm]{geometry}

\newcounter{iii}

\newcommand{\bb}{{\mathcal B}}
\newcommand{\aaa}{{\mathcal A}}

\newcommand{\ff}{\mathcal F}
\theoremstyle{plain}
\newtheorem{thm}{Theorem}
\newtheorem{stat}{Statement}

\newtheorem{prop}[thm]{Proposition}

\newtheorem{pro}{Problem}
\newtheorem{cor}[thm]{Corollary}
\theoremstyle{definition}

\numberwithin{equation}{section}
\numberwithin{thm}{section}

\title{Degree versions of theorems on intersecting families via stability}
\author{Andrey Kupavskii}
\address{University of Oxford\\
Moscow Institute of Physics and Technology; Email: {\tt kupavskii@ya.ru}.} \thanks{The research was supported by the Advanced Postdoc.Mobility grant no. P300P2\_177839 of the Swiss National Science Foundation, by the Russian Foundation for Basic Research (grant
no.  18-01-00355) and the Council for the Support of Leading Scientific Schools of the President of the
Russian Federation (grant no.
N.Sh.-6760.2018.1).}
\date{}

\begin{document}
\maketitle
\begin{abstract} For a family of subsets of an $n$-element set, its matching number is the maximum number of pairwise disjoint sets. Families with matching number $1$ are called intersecting. The famous Erd\H os--Ko--Rado theorem determines the size of the largest intersecting family of $k$-sets. The problem of determining the largest family of $k$-sets with matching number $s>1$ is known under the name of  the Erd\H os Matching Conjecture and is still open for a wide range of parameters. In this paper, we study the degree versions of both theorems.

More precisely, we give degree and  $t$-degree versions of the Erd\H os--Ko--Rado and the Hilton--Milner theorems, extending the results of Huang and Zhao, and Frankl, Han, Huang and Zhao. We also extend the range in which the degree version of the Erd\H os Matching Conjecture holds.

\end{abstract}

\section{Introduction}\label{sec1}
For integers $a\le b$, put $[a,b]:=\{a,a+1,\ldots, b\}$ and $[n]:=[1,n]$. For a set $X$ and an integer $k\ge 0$, let  $2^X$ and ${X\choose k}$ stand for the collections of all subsets and of all $k$-element subsets ({\it $k$-sets}) of $X$, respectively. Any collection of sets is called a {\it family}.  We call a family {\it intersecting}, if any two sets from it intersect. A ``trivial'' example of such family is all sets containing a fixed element. A family $\ff$ is called \textit{non-trivial} if $\bigcap _{F\in \ff}F=\emptyset$.

The following theorem is one of the classic results in extremal combinatorics.
\begin{thm}[Erd\H os, Ko, Rado \cite{EKR}]\label{thmekr} Let $n\ge 2k>0$ and consider an intersecting family $\ff\subset {[n]\choose k}$. Then $|\ff|\le {n-1\choose k-1}$.  \end{thm}

Answering a question of Erd\H os, Ko, and Rado, Hilton and Milner \cite{HM} determined the largest non-trivial intersecting family of $k$-sets. If $k>3$ then, up to a permutation of the ground set, the unique largest non-trivial intersecting family is $\mathcal H_{k}$, where for any  $u\in[2,k+1]$
$$\mathcal H_u:=\Big\{A\in {[n]\choose k}:[2,u+1]\subset A\Big\}\cup\Big\{A\in{[n]\choose k}: 1\in A, [2,u+1]\cap A\ne \emptyset\Big\}.$$
For $k=3$, $\mathcal H_3$ is the largest but not unique: $\mathcal H_2$ has the same size.

Erd\H os--Ko--Rado theorem spurred the development of extremal combinatorics, and by now there are numerous variations and extensions of Theorem~\ref{thmekr} and the Hilton--Milner theorem (cf. \cite{BGPR, BNR, Bor, Bor2, Derevo, EKL, FK5, FK1, HK, IK, KL, KostM, Kup21, KZ, KSV, Pyad2, Pyad, PR, Raig, WZ} to name a few recent ones). We refer the reader to a recent survey by Frankl and Tokushige~\cite{FT7}.

For $\ff\subset {X\choose k}$ and $i\in X$, the {\it degree} $d_i(\ff)$ of an element $i$ is the number of sets from $\ff$ containing it. We denote by $\delta(\ff)$ and $\Delta(\ff)$  the {\it minimum degree} and {\it maximum degree} of an element in $\ff$. Recently, Huang and Zhao \cite{HZ} gave an elegant proof of the following theorem using a linear-algebraic approach:

\begin{thm}[\cite{HZ}]\label{thmhz} Let $n>2k>0$. Then any intersecting family $\ff\subset {[n]\choose k}$ satisfies  $\delta(\ff)\le {n-2\choose k-2}$.
\end{thm}
The bound in the theorem is tight because of the trivial intersecting family, and the condition $n>2k$ is necessary: the authors of \cite{HZ} provide an example of such family for $n=2k$ which has larger minimum degree. In fact, for most values of $k$ there are {\it regular} intersecting families in ${[2k]\choose k}$ of maximum possible size ${2k-1\choose k-1}$ (see \cite{IK}).  In the follow-up paper, Frankl, Han, Huang, and Zhao \cite{FHHZ} proved the following theorem.

\begin{thm}[\cite{FHHZ}]\label{thmfhhz} Let $k\ge 4$ and $n\ge ck^2$, where $c=30$ for $k=4,5$, and $c=4$ for $k\ge 6$. Then any non-trivial intersecting family $\ff\subset{[n]\choose k}$ satisfies $\delta(\ff)\le {n-2\choose k-2}-{n-k-2\choose k-2}$.
\end{thm}
This theorem is again tight: the lower bound is provided by the Hilton--Milner family $\mathcal H_k$.

 Several questions and problems in this context were asked in \cite{HZ} and \cite{FHHZ}, as well as in personal communication with Hao Huang and his presentation on the Recent Advances in Extremal Combinatorics Workshop at TSIMF, Sanya, (May 2017). Some of them are as follows:
\setlist{leftmargin=1cm}
\begin{enumerate}
\item Can one find a combinatorial proof of Theorem~\ref{thmhz}? This question was partially answered by Frankl and Tokushige \cite{FT6}, who proved it under the additional assumption $n\ge 3k$. Huang claims that their proof can be made to work for $n\ge 2k+3$, provided that one applies their approach more carefully. However, the cases $n=2k+2$ and $n=2k+1$ remained open.
\item Extend Theorem~\ref{thmfhhz} to the case $n\ge ck$ for large $k$. Ultimately, prove Theorem~\ref{thmfhhz} for all values $n\ge 2k+1$ for which it is valid.

\item Extend Theorems~\ref{thmhz} and \ref{thmfhhz} to {\it $t$-degrees}. The {\it degree of a subset} $S\subset [n]$ is the number of sets from the family containing $S$. We denote by $\delta_t(\ff)$ the minimal degree of a $t$-element subset $S\subset [n]$ ({\it minimal $t$-degree}).
\end{enumerate}

In this note, we partially answer these three questions. Our first theorem provides a $t$-degree version of Theorem~\ref{thmhz}. Its proof is combinatorial and works, in particular, for $t=1$ and $n\ge 2k+2$.

\begin{thm}\label{thm01} If $n\ge 2k+2$, then for any intersecting family $\ff$ of $k$-subsets of $[n]$ we have $\delta(\ff)\le {n-2\choose k-2}$. More generally, if $n\ge 2k+\frac{3t}{1-\frac tk}$ and $1\le t<k$, then $\delta_t(\ff)\le {n-t-1\choose k-t-1}$.\end{thm}
After writing a preliminary version of the paper, we read the paper \cite{FT6}, where Theorem~\ref{thm01} is proved for $t=1$ and $n\ge 3k$. It turned out that the approach the authors took is very similar to the approach we use to prove Theorem~\ref{thm01}. However, it seems that their proof, unlike ours, does not work for $n=2k+2$, which is probably due to the fact that they use the original Frankl's degree theorem \cite{Fra1} instead of Theorem~\ref{thmkz} (cf. Section~\ref{secprel}).

Our main theorem is a $t$-degree version of Theorem~\ref{thmfhhz} with much weaker restrictions on $n$ (for moderately large $k$).
\begin{thm}\label{thm02} If $t=1$, $n\ge 2k+5$, and $k\ge 30$, or $1<t\le \frac k4-2$, $n\ge 2k+14t$, then for any non-trivial intersecting family $\ff$ of $k$-subsets of $[n]$ we have $\delta_t(\ff)\le {n-t-1\choose k-t-1}-{n-t-k-1\choose k-t-1}$. \end{thm}

The dependencies in the theorem are not optimal. One reason is that there is a tradeoff between different parameters. (E.g., it follows from the proof that the theorem holds for $n\ge 2k+6$ and $k\ge 15$.) Most importantly, pushing farther the bounds on $n,k$ would require much more complicated calculations, which we decided to avoid. It is very much possible that one can show the validity of Theorem~\ref{thm02} for $n\ge 2k+3$ and, say, $k\ge 100$, using a refinement of our approach. We made even less effort to optimize the dependencies for $t>1$. Still, the condition on $n$ here is much weaker than that in Theorem~\ref{thmfhhz} and is very close to the best possible one.

We are unaware of any examples that show that Theorem~\ref{thm01} or~\ref{thm02} does not hold for $n>2k$ with some $t\ge 1$. We suspect that both theorems should be valid for $n>2k+c$ with some fixed small $c$ (or even $c=1$) for {\it any} $t$.\\

The {\it matching number} $\nu(\ff)$ of a family $\ff$ is the maximum number of pairwise disjoint sets from $\ff$. That is, the intersecting families are exactly the families with matching number $1$. It is a natural question to ask, what is the largest family with matching number (at most) $s$. Let us denote by $m(n,k,s)$ the size of the largest family $\ff\subset{[n]\choose k}$ with $\nu(\ff)\le s$. Note that this question is only interesting for $n\ge k(s+1)$.  Consider the following two families: $$\mathcal A_0(n,k,s):=\{A\subset[n]: A\cap [s]\ne \emptyset\},$$
$$A_k(k,s):={[k(s+1)-1]\choose k}.$$
Erd\H os conjectured \cite{E} that $m(n,k,s)=\max\big\{|\aaa_0(n,k,s)|,|\aaa_k(k,s)|\big\}$. This conjecture is known under the name of the Erd\H os Matching Conjecture. It was studied quite extensively over the last 50 years, but still remains unsolved in general.  It is known to be true for $k\le 3$ \cite{F5}, for $n\ge (2s+1)k-s$ \cite{F4}, as well as for $n\ge \frac 53sk-\frac 23s$ for $s\ge s_0$ \cite{FK16}. We note that $\aaa_0(n,k,s)$ is bigger than $\aaa_k(k,s)$ already for relatively small $n$: the condition $n>(k+1)(s+1)$ suffices.

Finding degree versions of the Erd\H os Matching Conjecture falls into a more general recent trend to study the so-called Dirac thresholds.\footnote{By analogy with the famous degree criterion for the existence of a Hamilton cycle in a graph.} See, e.g.,  \cite{aletal}, \cite{FK16}, \cite{HPS}, \cite{KO}.
The following theorem was proved in \cite{HZ}.
\begin{thm}[\cite{HZ}]\label{thmhzm} Given $n,k,s$ with $n\ge 3k^2(s+1)$, if for a family $\ff\subset{[n]\choose k}$ with $\nu(\ff)\le s$ we have $\delta_1(\ff)\le {n-1\choose k-1}-{n-s-1\choose k-1}=\delta_1(\aaa_0(n,k,s))$.
\end{thm}
This improved the result of Bollob\'as, Daykin, and Erd\H os \cite{BDE}, who arrived at the same conclusion for $n\ge 2k^3s$. The authors of \cite{HZ} conjectured that the same should hold for any $n>k(s+1)$.
Note that the family $\aaa_k(k,s)$ does not appear in the degree version since its minimum $t$-degree is 0 for $n\ge k(s+1)$ and $t\ge 1$.  Note that, for integer $t\ge 1$, we have $\delta_t(\aaa_0(n,k,s))= {n-t\choose k-t}-{n-s-t\choose k-t}$.

In this paper we improve and generalize Theorem~\ref{thmhzm} for $k$ large in comparison to $s$.

\begin{thm}\label{degemc} Fix $n,s,k$ and $t\ge 1$, such that $n\ge 2k^2$,  and $k\ge 5st$ ($k\ge 3s$ for $t=1$). For any family $\ff\subset {[n]\choose k}$ with $\nu(\ff)\le s$ we have $\delta_t(\ff) \le \delta(\aaa_0(n,k,s))$, with equality only in the case $\ff=\aaa_0(n,k,s)$.
\end{thm}
The constants here  are again not optimal and chosen in this way to simplify the calculations.

All theorems are proved using stability results for the corresponding problems (cf. Section~2). We also note that  one can obtain asymptotic analogues of Theorems~\ref{thmhzm} and~\ref{degemc}  for fixed $k$ and $n>\frac 53ks$ from the results on the EMC that we mentioned (cf. \cite{FK16} for details).

\section{Preliminaries}\label{secprel}

For a family $\ff$, the {\it diversity} $\gamma(\ff)$ is the quantity $|\ff|-\Delta(\ff)$.  In particular, a family $\ff$ is non-trivial if and only if $\gamma(\ff)\ge 1$. The following far-reaching generalization of the Hilton--Milner theorem was proved by Frankl \cite{Fra1} and further strengthened by Kupavskii and Zakharov \cite{KZ}.
\begin{thm}[\cite{KZ}]\label{thmkz} Let $n>2k>0$ and $\ff\subset {[n]\choose k}$ be an intersecting family. If $\gamma(\ff)\ge {n-u-1\choose n-k-1}$ for some real $3\le u\le k$, then \begin{equation}\label{eq01}|\ff|\le {n-1\choose k-1}+{n-u-1\choose n-k-1}-{n-u-1\choose k-1}.\end{equation}
\end{thm}

The bound from Theorem \ref{thmkz} is sharp for integer $u$, as witnessed by $\mathcal H_u$.

In \cite{Kupstruc}, the author derived the following handy corollary of some general results in the spirit of Theorem~\ref{thmkz}.

\begin{cor}\label{corweight}
 Let $n>2k\ge 6$. For any intersecting family $\ff\subset{[n]\choose k},$ $\gamma(\ff)\le {n-4\choose k-3}$, we have $|\Delta(\ff)|+\frac{n-k-2}{k-2}\gamma(\ff)\le {n-1\choose k-1}$.
\end{cor}

In the proof of Theorem~\ref{degemc}, we use the following stability theorem, proved by Frankl and the author \cite{FK7, FK8}. Recall that {\it covering number} $\tau(\ff)$ is the minimal size of a set $S\subset [n]$, such that $S\cap F\ne \emptyset$ for any $F\in \ff$. Among the families $\ff\subset {[n]\choose k}$  that  satisfy $\nu(\ff)\le s$, the families with $\tau(\ff)>s$ are exactly the ones that are not isomorphic to a subfamily of $\mathcal A_0(n,k,s)$.

\begin{thm}[\cite{FK7}]\label{thmhil2} Fix integers $s,k\ge 2$. Let $n = (u+s-1)(k-1)+s+k,$ $u\ge s+1$. Then for any family $\mathcal G\subset {[n]\choose k}$ with $\nu(\mathcal G)= s$ and $\tau(\mathcal G)\ge s+1$ we have
\begin{equation}\label{eqhil} |\mathcal G|\le {n\choose k}-{n-s\choose k} - \frac{u-s-1}u{n-s-k\choose k-1}.
\end{equation} \end{thm}\vskip+0.4cm

\subsection{Calculations}\label{sec50}

In this section we do some of the calculations used in the proofs of Theorems~\ref{thm01} and~\ref{thm02}. Substituting $u=3$ in \eqref{eq01}, we get that \begin{align}\notag|\ff|\ \le&\ {n-1\choose k-1}-{n-4\choose k-1}+{n-4\choose k-3} = \sum_{i=2}^4{n-i\choose k-2}+{n-4\choose k-3}\\ \notag =&\ {n-2\choose k-2}+2{n-3\choose k-2}=\Big(1+\frac{2(n-k)}{(n-2)}\Big){n-2\choose k-2}\\ \label{eq25}=&\ \frac{k(k-1)(3n-2k-2)}{n(n-1)(n-2)}{n\choose k}.\end{align}

We also have
$$\frac{{n-u-1\choose n-k-1}}{{n-u-1\choose k-1}}=\frac{\prod_{i=1}^{n-k-1}\frac{n-u-i}i} {\prod_{i=1}^{k-1}\frac{n-u-i}i} = \prod_{i=k}^{n-k-1}\frac{n-u-i}i= \prod_{i=k}^{n-k-1}\frac{n-u-i}{n-1-i}.$$
Clearly, in the range $3\le u\le k$ the last expression is maximized for $u=3$, and we get the following equality, provided $n\ge 2k+2$: 
\begin{equation}\label{eq04}\frac{{n-u-1\choose n-k-1}}{{n-u-1\choose k-1}}\le \prod_{i=k}^{n-k-1}\frac{n-3-i}{n-1-i}\le \frac{(k-1)(k-2)}{(n-k-1)(n-k-2)}\ \ \ \ \ \ \text{for }3\le u\le k.\end{equation}

We will also use the following formula:
\begin{equation}\label{eq07}
\frac{{n-t-k-1\choose k-t-1}}{{n-t-1\choose k-t-1}}=\frac{n-2k+1}{n-t-k}\cdot\frac{{n-t-k\choose k-t-1}}{{n-t-1\choose k-t-1}}=\ldots = \prod_{i=1}^{k}\frac{n-k+1-i}{n-t-i}.
\end{equation}

\section{Proofs}
\subsection{Proof of Theorem~\ref{thm01}}\label{sec51}

Take an intersecting family $\ff$ with maximum degree $\Delta$ and diversity $\gamma$. Then, by definition, $|\ff|= \Delta+\gamma$. W.l.o.g., we suppose that $\Delta(\ff)=d_1(\ff)$. The statement is vacuously true for trivial intersecting families, so we may assume that $\gamma\ge 1$.  We have two cases to distinguish.\vskip+0.1cm

\textbf{Case 1. $\gamma\le {n-4\choose k-3}$. }  In this case we use the following proposition.
\begin{prop}\label{stat1} Fix some $n,t,k$. If for an intersecting family of $k$-sets $\ff\subset 2^{[n]}$ with maximum degree $\Delta$ and diversity $\gamma$  we have \begin{equation}\label{eq02}\Delta+\frac k{k-t}\gamma\le {n-1\choose k-1},\end{equation} then $\delta_t(\ff)\le {n-t-1\choose k-t-1}$.\end{prop}

\begin{proof}
The sum of $t$-degrees of all $t$-subsets of $[2,n]$ is $\gamma {k\choose t}+\Delta {k-1\choose t}.$ Therefore, there is a $t$-tuple $T$ of elements in $[2,n]$, such that
\begin{equation}\label{eq4} \delta_t(T)\le \frac{\gamma {k\choose t}+ \Delta {k-1\choose t}}{{n-1\choose t}}=\frac{{k-1\choose t}}{{n-1\choose t}}\Big(\frac k{k-t}\gamma+\Delta\Big)\overset{\eqref{eq02}}{\le} \frac{{k-1\choose t}{n-1\choose k-1}}{{n-1\choose t}}={n-t-1\choose k-t-1}.\end{equation}
\end{proof}

To prove the theorem in Case~1, it is sufficient to  verify \eqref{eq02} for all intersecting families.
We may apply Corollary~\ref{corweight} to $\ff$ (otherwise, it is not difficult to obtain via direct calculations from \eqref{eq01}). We only have to check that \begin{equation}\label{eq151}\frac{k}{k-t}\le \frac{n-k-2}{k-2}\ \ \ \ \Leftrightarrow \ \ \ \ \frac{t}{k-t}\le \frac {s}{k-2},\end{equation} where $n = 2k+s$ and $s\ge 1$. We see that if $t=1$, then \eqref{eq151} holds for any $s\ge 1$. If $t>1$, then we must have $$s\ge \frac{k-2}{k-t}t \ \ \ \Leftarrow \ \ \  s\ge \frac{t}{1-\frac tk}.$$



\textbf{Case 2. $\gamma\ge {n-4\choose k-3}$. } We use the following bound on $\delta_t(\ff):$ \begin{equation}\label{eq06}\delta_t(\ff)\le \frac{{k\choose t}}{{n\choose t}} |\ff|.\end{equation}
Thus, it is sufficient for us to check that the following inequality holds:
\begin{equation}\label{eq05}|\ff|\le \frac{{n\choose t}{n-t-1\choose k-t-1}}{{k\choose t}}=\frac{k-t}{n-t}{n\choose k}.\end{equation}
Combining \eqref{eq01} and \eqref{eq25}, we get that \eqref{eq05} is satisfied if
$$ \frac{k(k-1)(3n-2k-2)}{n(n-1)(n-2)}\le \frac{k-t}{n-t}.$$
If $t=1$ then the expression above simplifies to $\frac{k(3n-2k-2)}{n(n-2)}\le 1$, which holds for any $n\ge 2k+2$.
If $t>1$, then,  knowing that $n=k$ is a root of the expression above,  it simplifies to the following quadratic inequality\footnote{One simply has to multiply both sides by the denominators, move everything to the right side and divide the expression by $n-k$.} in $n$: $(k-t)n^2 -(2k^2 +(k-3)t)n +2(k^2-1)t\ge 0$, which is definitely valid if
$$n\ge \frac{2k^2 +kt}{k-t}=2k+\frac{3t}{1-t/k}.$$


\subsection{Proof of Theorem~\ref{thm02}}\label{sec52}
The strategy of the proof is very similar to that of Theorem~\ref{thm01}. Fix a non-trivial intersecting family $\ff$ with maximum degree $\Delta$ and diversity $\gamma\ge 2$ (if $\gamma\le 1$ then $\ff$ is a subfamily of either an Erd\H os-Ko-Rado or a Hilton--Milner family). W.l.o.g., suppose that $d_1(\ff)=\Delta(\ff)$ and that $\ff$ contains the set $[2,k+1]$. Then any other set containing $1$ must intersect $U$. We compare $\ff$ with the Hilton-Milner family $\mathcal H_k.$ We consider cases depending on $\gamma$.
The case analysis, however, is be more complicated, as compared to the previous case. Notably, we get a new non-trivial Case 1.\\

\textbf{Case 1. $1<\gamma<{n-k+t+1\choose t+2}$. }
 We have $\gamma\ge 2$, and thus there is $U'\in\ff$, such that $1\notin U'$ and $U'\ne [2,k+1]$. Then it is easy to see that\footnote{We may either use the Kruskal--Katona theorem in terms of cross-intersecting families, or do the following direct computation. If $|U\cap U'|=t$, $1\le t\le k-1$, then the number of sets containing $1$ and intersecting both $U$ and $U'$ is ${n-1\choose k-1}-2{n-k-1\choose k-1}+{n-2k+t-1\choose k-1}$. This is clearly maximized for $t=k-1$, in which case we get that there are ${n-1\choose k-1}-2{n-k-1\choose k-1}+{n-k-2\choose k-1} = {n-1\choose k-1}-{n-k-1\choose k-1}-{n-k-2\choose k-2}$ such sets.} $d_1(\ff)\le {n-1\choose k-1}-{n-k-1\choose k-1}-{n-k-2\choose k-2}$. To put it differently, we have $d_1(\mathcal H_k)-d_1(\ff)\ge {n-k-2\choose k-2}$.
Let us denote by $\alpha$ the number of sets containing $1$ and intersecting $[k+2,n]$ in at most $t-1$ elements. Then
\begin{equation}\label{eq18}\delta_t(\mathcal H_k)-\delta_t(\ff)\ge \frac{{n-k-2\choose k-2}-\alpha-{k\choose t}\gamma}{{n-k-1\choose t}}.
\end{equation}
 The right hand side actually provides a lower bound for the difference of the {\it average} degree in $[k+2,n]$ in $\mathcal H_k$ and $\mathcal F$. Since the average $t$-degree on $[k+2,n]$ in $\mathcal H_k$ is equal to the minimum $t$-degree, the right hand side is also a lower bound for the difference of the minimum $t$-degree. Let us show that the right hand side is a lower bound for the difference in the average $t$-degrees on $[k+2,n]$. Indeed, the first two terms in the numerator on the right hand side gives a lower bound on the average loss in the $t$-degree of $t$-sets in $[k+2,n]$ due to the missing sets that contain $1$ and intersect $[k+2,n]$ in at least $t$ elements. Each of them contribute at least $1$ to the $t$-degree of at least one $t$-set in $[k+2,n]$. At the same time, each set contributing to $\gamma$ can contribute at most ${k\choose t}/{n-k-1\choose t}$ to the average $t$-degree on $[k+2,n]$, which explains the third term.

We have
$$\alpha=\sum_{i=k-t}^{k-1}{k\choose i}{n-k-1\choose k-i-1}.$$
We have ${n-k-1\choose t-j}>2{n-k-1\choose t-j-1}$ for any $j\ge0$, since $n> 2k\ge 6t$.  Therefore, $$\alpha\le {k\choose k-t}{n-k-1\choose t}={k\choose t}{n-k-1\choose t}.$$
To show that the RHS in \eqref{eq18} is always nonnegative and thus to conclude the proof in Case 1, it is sufficient to show the following inequality:
\begin{equation}\label{eq19} {n-k-2\choose k-2}\ge {k\choose t}\Big({n-k-1\choose t}+{n-k+t+1\choose t+2}\Big).\end{equation}
Note that we used the assumption $\gamma< {n-k+t+1\choose t+2}$. In what follows, we verify \eqref{eq19}.\\

If $t=1$, $n\ge 2k+5$, then in the worst case for \eqref{eq19} is $n=2k+5$. This reduces to  $${k+3\choose 5}\ge k\Big(k+4+{k+7\choose 3}\Big),$$
which holds for $k\ge 30$.\footnote{If $n\ge 2k+6$, then \eqref{eq19} holds already for $k\ge 15$, and if $n\ge 2k+8$, then it holds for $k\ge 10$.}\\

If $1<t\le \frac k{4}-2$, $n\ge 2k+14t$, then the right hand side of \eqref{eq19} is at most ${k\choose t}{n-k+t+2\choose t+2}$, and we have \begin{align}\notag {k\choose t}{n-k+t+2\choose t+2} =&\  \frac{k!(n-k+t+2)!}{(k-t)!t!(t+2)!(n-k)!}\\ \notag\le& \ \Big(\frac{n-k-2}{n-k-t-4}\Big)^{t+4}\cdot \frac{k!(n-k-2)!}{(k-t)!(2t+2)!(n-k-t-4)!}\cdot {2t+2\choose t}\\ \notag
\le& \ \Big(\frac 54\Big)^{t+4}\frac{(n-k-2)!}{(2t+2)!(n-k-2t-4)!}\cdot 2^{2t}\\ \label{eqcomp}
\le&\ 5^{t+1}{n-k-2\choose 2t+2}.\end{align}

 On the other hand, in the same assumptions, we have \begin{align}\notag{n-k-2\choose 2t+2}<&\  \Big(\frac{4t+4}{n-k-4t-6}\Big)^{2t+2}{n-k-2\choose 4t+4} \\ \notag \le&\  \Big(\frac{4t+4}{14t+2}\Big)^{2t+2}{n-k-2\choose 4t+4}\le \Big(\frac{2}{5}\Big)^{2t+2}{n-k-2\choose 2t+2}\\ \label{eqcomp2}<&\ 5^{-t-1}{n-k-2\choose k-2}.\end{align}
 Comparing \eqref{eqcomp} and \eqref{eqcomp2}, we conclude that \eqref{eq19} holds.







\vskip+0.2cm


\textbf{Case 2. ${n-k+t+1\choose t+2}\le \gamma\le {n-4\choose k-3}$. }
We can get the  following analogue of Statement~\ref{stat1}.
\begin{stat} If a family $\ff\subset {n\choose k}$ with maximum degree $\Delta$ and diversity $\gamma$  satisfies \begin{equation}\label{eq12}\Delta+\frac {k}{k-t}\gamma\le \Big(1-\prod_{i=1}^{k}\frac{n-k+1-i}{n-t-i}\Big){n-1\choose k-1},\end{equation} then $\delta_t(\ff)\le {n-t-1\choose k-t-1}-{n-t-k-1\choose k-t-1}$.\end{stat}
\begin{proof}
We may repeat the calculation in \eqref{eq4}, and get that
$$\delta_t(\ff)\le  \Big(1-\prod_{i=1}^{k}\frac{n-k+1-i}{n-t-i}\Big){n-t-1\choose k-t-1}  \overset{\eqref{eq07}}{=} {n-t-1\choose k-t-1}-{n-t-k-1\choose k-t-1}.$$
\end{proof}

We apply \eqref{eq01}. The bounds on $\gamma$ defining Case~2 correspond to the range $3\le u\le k-t-2$ in \eqref{eq01}. Then \eqref{eq12} is implied by the following inequality. \begin{equation}\label{eq13}{n-u-1\choose k-1}-\frac{k}{k-t}{n-u-1\choose n-k-1}-\prod_{i=1}^{k}\frac{n-k+1-i}{n-t-i}{n-1\choose k-1}\ge 0.\end{equation} We have
\begin{footnotesize}$${n-u-1\choose k-1}-\frac{k}{k-t}{n-u-1\choose n-k-1} \overset{\eqref{eq04}}{\ge} \Big(1- \frac{k(k-1)(k-2)}{(k-t)(n-k-1)(n-k-2)}\Big) \prod_{i=1}^{k-1}\frac{n-u-i}{n-i}{n-1\choose k-1}.$$\end{footnotesize}
The last expression is minimized when $u=k-t-2$. Comparing the product above with the product in \eqref{eq13}, we get \begin{small}$$\frac{\prod_{i=1}^{k}\frac{n-k+1-i}{n-t-i}} {\prod_{i=1}^{k-1}\frac{n-u-i}{n-i}}\le \frac{\prod_{i=1}^{k}\frac{n-k+1-i}{n-t-i}} {\prod_{i=1}^{k-1}\frac{n-k+t+2-i}{n-i}}= \frac{\prod_{i=1}^{k-t-1}\frac{n-k-t-i}{n-t-i}} {\prod_{i=1}^{k-t-2}\frac{n-k+1-i}{n-i}}\le \frac{\prod_{i=1}^{k-t-1}\frac{n-k-t-i}{n-t-i}} {\prod_{i=1}^{k-t-2}\frac{n-k-t-i}{n-t-i}}=1-\frac{k}{n-k+1}.$$\end{small}
Therefore, to prove \eqref{eq13}, it is sufficient for us to show that \begin{equation}\label{eq16} 1- \frac{k(k-1)(k-2)}{(k-t)(n-k-1)(n-k-2)}\ge 1-\frac{k}{n-k+1}.\end{equation}
For the fraction on the left hand side, we use the following property: if we add 1 to one of the multiples in the numerator and 1 to one of the second or third multiple in the denominator, then the fraction will only increase, and the expression in the left hand side will decrease. If $t=1$, then the LHS of \eqref{eq16} is $$1- \frac{k(k-2)}{(n-k-1)(n-k-2)}\ge 1-\frac{k^2}{(n-k+1)(n-k-2)}\ge 1-\frac{k}{n-k+1},$$
where the last inequality holds s long as $n\ge 2k+2$. Therefore, \eqref{eq13} is satisfied for any $k$ and $n\ge 2k+2$.

If $1<t\le \frac k4$ and $n\ge 2k+4t$, then $(k-t)(n-k-1)\ge (k-t)(k+4t-1)= k^2+3kt-4t^2-k+t>k^2$, and, therefore,
the LHS of \eqref{eq16} is at least $$1- \frac{k^3}{(k-t)(n-k-1)(n-k+1)}> 1-\frac{k}{n-k+1},$$
and \eqref{eq13} is satisfied again.\\

\textbf{Case 3. $\gamma\ge {n-4\choose k-3}$. } We again use the bound \eqref{eq06}. That is, we have to verify that
\begin{equation}\label{eq08}|\ff|\le \frac{{n\choose t}\big({n-t-1\choose k-t-1}-{n-t-k-1\choose k-t-1}\big)}{{k\choose t}}\overset{\eqref{eq07}, \eqref{eq05}}{=}\frac{k-t}{n-t}\Big(1-\prod_{i=1}^{k}\frac{n-k+1-i}{n-t-i}\Big){n\choose k}.\end{equation}

Using \eqref{eq25}, the inequality \eqref{eq08} is implied by
\begin{equation}\label{eq11}\frac{k-t}{n-t} \Big(1-\prod_{i=1}^{k}\frac{n-k+1-i}{n-t-i}\Big)\ge \frac{k(k-1)(3n-2k-2)}{n(n-1)(n-2)}.\end{equation}
In what follows, we verify \eqref{eq11}.
If $t=1$, then it simplifies to $$\Big(1-\prod_{i=1}^{k}\frac{n-k+1-i}{n-1-i}\Big)\ge \frac{k(3n-2k-2)}{n(n-2)} \ \ \Leftrightarrow \ \ \frac{(n-k)(n-2k-2)}{n(n-2)}\ge \prod_{i=1}^{k}\frac{n-k+1-i}{n-1-i}.$$
This is equivalent to \begin{equation}\label{eq09}\frac{(n-2k-2)}{n}\ge \prod_{i=2}^{k}\frac{n-k+1-i}{n-1-i}.\end{equation}
The right hand side is at most $$\frac{(n-2k+1)(n-2k+2)}{(n-3)(n-4)}\le \frac{(n-2k+4)(n-2k+2)}{n(n-4)}.$$ Therefore \eqref{eq09} follows from
$$\frac{n-2k-2}{n}\ge \frac{(n-2k+4)(n-2k+2)}{n(n-4)}\ \ \Leftrightarrow (n-2k-2)(n-4)\ge (n-2k+4)(n-2k+2),$$ which holds for any $n\ge 2k+4$ and $k\ge 12.$

If $1<t\le \frac k4-2$, then $\frac{k+\frac 43 t}k\ge \frac k{k-t}$. We use this in the second inequality below:
$$1-\frac{n-t}{k-t}\cdot\frac{k(k-1)(3n-2k-2)}{n(n-1)(n-2)}\ge 1-\frac{k^2(3n-2k)}{(k-t)n^2}\ge$$
$$1-\frac{(k+\frac 43t)(3n-2k)}{n^2} = \frac{(n-k)(n-2k)-\frac 43 t(3n-2k)}{n^2}\ge \frac{(n-k)(n-2k-6t)}{n^2}.$$
On the other hand,
$$\prod_{i=1}^{k}\frac{n-k+1-i}{n-t-i}=\prod_{i=1}^{k-t-1} \frac{n-t-k-i}{n-t-i}\le \Big(\frac{n-k}{n}\Big)^{k-t-1}\le \Big(\frac {n-k}n\Big)^{\frac {3k}{4}+1}.$$
Therefore, combining these calculations, the inequality \eqref{eq11} would follow from the inequality
\begin{equation}\label{eq10}1-\frac{2k+6t}n\ge \Big(1-\frac k{n}\Big)^{3k/4}.\end{equation}
 If $2k+14t\le n\le7k$, then the right hand side of the inequality above is at most $ e^{-\frac{3k^2}{4n}}<e^{-\frac {k}{10}},$ while the left hand side is at least $\frac {8t}{2k+14t}>\frac {16}{2k+28}$. It is easy to see that, say, for $k\ge 15$, we have $\frac {16}{2k+28}>e^{-\frac {k}{10}}$.

If $n>7k$ and $k\ge10$, then $$\Big(1-\frac k{n}\Big)^{3k/4}<\Big(1-\frac k{n}\Big)^{7}<1-\frac {7k}n+\frac{21k^2}{n^2}\le 1-\frac {4k}n <1-\frac {2k+6t}n.$$
Therefore, the inequality \eqref{eq10} is verified.

To conclude, we remark that the only conditions on $k$ that we used for $t\ge 2$ were $k\ge 4t+8$ and $k\ge 10$. The later one is implied by the former one.

\subsection{Proof of Theorem \ref{degemc}}\label{sec6}
Fix $t\ge 1$ and take a family $\ff$ satisfying the requirements of the theorem. If $\ff$ is isomorphic to a subfamily of $\aaa_0(n,k,s)$ then we are done. Otherwise, $\tau(\ff)\ge s+1$ and thus  $|\ff|$ satisfies \eqref{eqhil}.
By simple double counting as in \eqref{eq06}, we have
$$\delta_t(\ff)\le \frac{{k\choose t}}{{n\choose t}}\Big[{n\choose k}-{n-s\choose k} - \frac{u-s-1}u{n-s-k\choose k-1}\Big].$$
Note that $$\delta_t(\mathcal A_0(n,k,s)) = {n-t\choose k-t}-{n-s-t\choose k-t} = \frac{{k\choose t}}{{n\choose t}}{n\choose k}-\frac{{k\choose t}}{{n-s\choose t}}{n-s\choose k}.$$
Therefore,
$$\delta_t(\aaa_0(n,k,s))-\delta_t(\ff)\ge \frac{{k\choose t}(u-s-1)}{{n\choose t}u}{n-s-k\choose k-1}-\Big[\frac{{k\choose t}}{{n-s\choose t}}-\frac{{k\choose t}}{{n\choose t}}\Big]{n-s\choose k}=$$
$$\frac{{k\choose t}}{{n\choose t}}\Bigg[\frac{u-s-1}{u}{n-s-k\choose k-1}-\Big[\prod_{i=0}^{t-1} \frac{n-i}{n-s-i}-1\Big]{n-s\choose k}\Bigg]=:(*).$$
We have $\prod_{i=0}^{t-1} \frac{n-i}{n-s-i}-1\le (1+\frac s{n-s-t})^t-1.$ It is not difficult to verify that for $\theta<\frac 1{2m}$ one has $(1+\theta)^m\le 1+2m\theta$. Therefore, assuming that \begin{equation}\label{eq33}n\ge s+t+2st,\end{equation} we have \begin{equation}\label{eq35}\prod_{i=0}^{t-1} \frac{n-i}{n-s-i}-1\le \frac {2ts}{n-s-t}.\end{equation}
On the other hand, we have $(1-\theta)^m\ge 1-m\theta$ and $t\le k-1$, and thus $$\frac{{n-s-k\choose k-1}}{{n-s\choose k}}=\frac {k}{n-s-k+1}\prod_{i=0}^{k-2}\frac{n-s-k-i}{n-s-i}>\frac{k}{n-s-t} \Big(1-\frac {k}{n-s-k}\Big)^{k-1}>$$
$$\frac{k}{n-s-t}\Big(1-\frac{(k-1)k}{n-s-k}\Big)>\frac k{2(n-s-t)},$$
provided \begin{equation}\label{eq34}n\ge s+k+2k(k-1).\end{equation} We conclude that, provided that \eqref{eq33} and \eqref{eq34} hold, we get
$$(*)> \frac{{k\choose t}}{{n\choose t}}\Bigg[\frac{u-s-1}{u}\frac k{2(n-s-t)}-\frac {2ts}{n-s-t}\Bigg]{n-s\choose k},$$
which is nonnegative provided $k\ge 4ts\frac u{u-s-1}$. This inequality holds for $k\ge 5ts$ and $u\ge 9s$. The last inequality is satisfied for $n\ge 2k^2$, since then $n\ge 2k^2\ge (9s+s)k\ge (9s+s-1)(k-1)+s+k$. We note that, with this choice of $n$ and $k$, both \eqref{eq33} and \eqref{eq34} hold.

For $t=1$ one may improve \eqref{eq35} to $\frac n{n-s}-1\le \frac s{n-s}$ and the condition on $k$ may be relaxed to $k\ge 3s$. The equality part of the statement follows easily from the fact that strict inequality is obtained in the case when $\tau(\ff)\ge s+1$.

\section{Conclusion}\label{sec7}
In this paper, we explored several questions concerning degrees and $t$-degrees of  intersecting families and families with small matchings. Some of these questions remain only partially resolved, and it would be highly desirable to settle them.

The first problem is to give a purely combinatorial proof of Theorem~\ref{thmhz} for $n=2k+1$. Although not fully optimized, our approach for the degree version of the Hilton--Milner family should surely fail for $n\le 2k+2$. Thus, we ask the following question.

\begin{pro} Is there an example for $n=2k+1$, such that there exists a non-trivial intersecting family $\ff$ with minimal $1$-degree higher than that of the Hilton-Milner family?
\end{pro}
One reason to believe that the answer to this question may be positive is that the degrees of elements in the Hilton-Milner family are irregular, even if we exclude the element of the highest degree out of consideration.
The main difficulty here is to understand the structure of families with large diversity for $n$ very close to $2k$.

We also believe that the conclusion of Theorems~\ref{thm01},~\ref{thm02} should hold for $n>2k+c$, where $c$ is some absolute constant, and any $t$.

Finally, the following question concerning cross-intersecting families seems interesting for us.

\begin{pro}\label{prob2} Consider two cross-intersecting families $\aaa,\ \bb\subset{[n]\choose k}$ that are disjoint. Is it true that $$\min\{|\aaa|, |\bb|\}\le \frac 12{n-1\choose k-1}?$$
\end{pro}

In between the first and the current version of the manuscript, Hao Huang published a manuscript \cite{Hua}, in which he addressed Problem~\ref{prob2}. In particular, he showed that the answer is positive as long as $n>2k^2$, and is negative as long as $n<ck^2$ for some $c>0$. Similar results were obtained by Frankl and the author \cite{FK17}.\\

{\sc Acknowledgements: } We thank Hao Huang for sharing the problem and for several fruitful discussions on the topic. We also thank the anonymous referees for many helpful comments that helped to correct some  calculation mistakes and improved the presentation of the paper.

\end{document}